%% file: main.tex
\documentclass[sigconf]{acmart}

\AtBeginDocument{%
  \providecommand\BibTeX{{%
    \normalfont B\kern-0.5em{\scshape i\kern-0.25em b}\kern-0.8em\TeX}}}

\include{macros}

\begin{document}

\title{The Constant Trace Property in Noncommutative Optimization}

\author{Ngoc Hoang Anh Mai}
\affiliation{%
  \institution{LAAS}
  \streetaddress{7 Avenue de Colonel Roche}
  \city{Toulouse}
  \country{France}
  \postcode{31400}}
\email{nhmai@laas.fr}

\author{Abhishek Bhardwaj}
\affiliation{%
  \institution{LAAS}
  \streetaddress{7 Avenue de Colonel Roche}
  \city{Toulouse}
  \country{France}
  \postcode{31400}}
\email{abhishek.bhardwaj@laas.fr}

\author{Victor Magron}
\affiliation{%
  \institution{LAAS}
  \streetaddress{7 Avenue de Colonel Roche}
  \city{Toulouse}
  \country{France}
  \postcode{31400}}
\email{victor.magron@laas.fr}

\renewcommand{\shortauthors}{Mai, Bhardwaj and Magron}

\begin{abstract}
In this article, we show that each semidefinite relaxation of a ball-constrained noncommutative polynomial optimization problem can be cast as a  semidefinite program with a constant trace matrix variable.
We then demonstrate how this constant trace property can be exploited via first order numerical methods to solve efficiently the  semidefinite relaxations of the noncommutative problem.
\end{abstract}
%

\keywords{noncommutative polynomial optimization, sums of hermitian squares, eigenvalue and trace optimization, conditional gradient-based augmented Lagrangian, constant trace property, semidefinite programming}


\maketitle

\newboolean{showext}
\setboolean{showext}{false}
\ifthenelse{\boolean{showext}}
{%
	\newcommand{\ext}[1]{{\color{blue}#1}}%
}%
{%
	\newcommand{\ext}[1]{}%
}

\input{Intro}
\input{Background}
\input{DenseEP}
\input{Sparsity}
\input{Experiments}
\input{Conclusion}

\bibliographystyle{abbrv}

\input{main.bbl}
\end{document}

%% file: macros.tex
\usepackage{mathrsfs}
\usepackage{enumerate} 
\usepackage{amsxtra,latexsym, amscd,amsthm}
\usepackage{indentfirst}
\usepackage{amsfonts}
\usepackage{graphics}
\usepackage{multirow}
\usepackage{array}
\newif\ifcomment
\commentfalse
\commenttrue
\usepackage{todonotes}

\usepackage{makecell}

\usepackage{verbatim}
\usepackage{algorithm}
\usepackage[noend]{algpseudocode}

\newcommand{\mc}{\mathcal}
\newcommand{\mf}{\mathfrak}
\newcommand{\mbb}{\mathbb}
\newcommand{\mbf}{\mathbf}

\newcommand{\norm}[1]{\left\lVert#1\right\rVert}
\newcommand{\abs}[1]{\lvert#1\rvert}
\newcommand{\set}[1]{\left\{#1\right\}}
\newcommand{\lr}[1]{\langle #1 \rangle}
\newcommand{\ceil}[1]{\lceil #1 \rceil}
\newcommand{\ul}[1]{\underline{#1}}

\DeclareMathOperator{\Sym}{Sym}
\DeclareMathOperator{\tr}{Tr}
\DeclareMathOperator{\normtr}{tr}

\def\R{{\mathbb R}}

\def\N{{\mathbb N}}

\def\Sym{\hbox{\rm{Sym}}}

\DeclareMathOperator{\I}{I}
\DeclareMathOperator{\II}{II}

\DeclareMathOperator{\diag}{diag}

\DeclareMathOperator{\var}{var}

\DeclareMathOperator{\s}{s}

\DeclareMathOperator{\cyc}{cyc}

\newtheorem{theorem}{Theorem}[section]
\newtheorem{corollary}[theorem]{Corollary}
\newtheorem{lemma}[theorem]{Lemma}
\newtheorem{proposition}[theorem]{Proposition}

\theoremstyle{definition}
\newtheorem{definition}[theorem]{Definition}

\theoremstyle{remark}
\newtheorem{remark}[theorem]{Remark}
\numberwithin{equation}{section}

%% file: Intro.tex
\section{Introduction}

\emph{Polynomial optimization problems} (POP) are present in many areas of mathematics, and science in general. There are many applications in global optimization, control and analysis of dynamical systems to name a few  \cite{lasserre2010moments}, and being able to efficiently solve POP is of great importance. 

In this article we focus on \emph{noncommutative (nc) polynomial optimization problems} (NCPOP), that is, polynomial optimization with non-commuting variables. NCPOP has several applications in control \cite{skelton1997} and quantum information  \cite{gribling2018,pal2009,marecek2020}.


Since the advent of interior point methods for \emph{semidefinite programs} (SDP) \cite{anjos2011handbook}, there have been many approaches to solving POP, using powerful representation results from \emph{real algebraic geometry} for positive polynomials. Inspired by Schm{\"u}dgen's solution to the \emph{moment problem} on compact semialgebraic sets \cite{schmudgen1991thek}, these methods aim to provide certificates of global positivity. There are natural analogues to these approaches in the nc setting, coming from \emph{free algebraic geometry} \cite{helton2013free}, and the \emph{tracial moment problems} \cite{burgdorf2012truncated}. 

A standard approach in the commutative setting, is \emph{Lasserre's Hierarchy} \cite{lasserre2001global}, which provides a sequence lower bounds on the optimal values for POPs, with guaranteed convergence under some natural constraints according to Putinar's Positivstellensatz \cite{Putinar1993positive}. 
This hierarchy and its nc extension to eigenvalue/trace optimization  \cite{pironio2010convergent,burgdorf2016optimization}, involve solving SDPs over the space of multivariate moment and nc Hankel matrices, respectively. 

Due to the current capacity of interior-point SDP solvers such as Mosek \cite{mosek2017mosek,andersen2000mosek}, these hierarchies can only be applied when the multivariate moment (or nc Hankel) matrices are of ``moderate'' size. Often restricting their use to polynomials of low degrees, or in few variables, with the situation being worse in the nc setting. 

A strategy for reducing the size of the SDP hierarchies is to exploit the sparsity structures of POPs. They include correlative sparsity (CS) in \cite{klep2019sparse} and term sparsity (TS), CS-TS in \cite{wang2020exploiting} all of which are the analogs of the commutative works about CS \cite{waki2006sums}, TS \cite{wang2021tssos,wang2021chordal} and CS-TS \cite{wang2020cs}.

Encouraged by \cite{helmberg2000spectral, yurtsever2019scalable}, in \cite{mai2020hierarchy, mai2020exploiting} the first and third authors showed how to exploit the \emph{Constant Trace Property} (CTP) for SDP relaxations of POPs, which is 
satisfied when the matrices involved in the SDP relaxations have constant trace.
By utilizing first order spectral methods to solve the required SDP relaxations, they attained significant computational gains for POPs constrained on simple domains, e.g., sphere, ball, annulus, box and simplex. 

In this article, we extend the exploitation of the CTP to NCPOPs. Our two main contributions are the following:
	First, we obtain analogous results to \cite{mai2020exploiting,mai2020hierarchy}, which ensure the CTP for a broad class of dense NCPOPs. 
	In particular, if nc ball (or nc polydisc) constraint(s) is present, then CTP holds.
	We also extend this CTP-framework to some NCPOPs with correlative sparsity.
	Secondly, We provide a Julia package for solving NCPOPs with CTP. The package makes use of first order methods for solving SDPs with CTP. We also demonstrate the numerical and computational efficiency of this approach, on some sample classes of dense NCPOPs and NCPOPs with correlative sparsity. 


%% file: Background.tex
\section{Definitions \& preliminaries}\label{sec:defs}

Here we introduce some basic preliminary knowledge needed in the sequel. For a more detailed introduction to the topics introduced in this section, the reader is referred to \cite{burgdorf2016optimization}.

\subsection{Noncommutative polynomials}
We denote by $\ul{X}$ the {noncommuting} 
letters $X_{1}, \dotsc, X_{n}$. 
Let $\lr{\ul{X}} = \lr{X_{1},\dotsc,X_{n}}$ be the \emph{free monoid} generated by $\ul{X}$, and call its elements \emph{words} in $\ul{X}$. 
Given a word $w=X_{i_1}\dots X_{i_r}$, $w^{*}$ is its reverse, i.e., $w^*=X_{i_r}\dots X_{i_1}$.
Consider the free algebra $\mbb{R}\lr{\ul{X}}$ of polynomials in $\ul{X}$ with coefficients in $\mbb{R}$. 
Its elements are called \emph{noncommutative (nc) polynomials}. 
Endow $\mbb{R}\lr{\ul{X}}$ with the involution $f\rightarrow f^{*}$ which fixes $\mbb{R}\cup\set{\ul{X}}$ pointwise. 
The length of the longest word in a polynomial $f\in\mbb{R}\lr{\ul{X}}$ is called the \emph{degree} of $f$ and is denoted $\deg(f)$. We write $\mbb{R}\lr{\ul{X}}_{d}$ for all nc polynomials of degree at most $d$. 
The set of symmetric elements of $\mbb{R}\lr{\ul{X}}$ is defined as $\Sym\mbb{R}\lr{\ul{X}} = \set{f\in\mbb{R}\lr{\ul{X}} : f^{*}=f}$. We employ the graded lexicographic ordering on all structures { and objects we consider.}

We write $\lr{\ul{X}}_{d}$ for the set of all words in $\lr{\ul{X}}$ of degree at most $d$, and we let $\mbf W_{d}(\ul{X}) {\equiv\mbf{W}_{d}}$  be the column vector of words in $\lr{\ul{X}}_{d}$, and $\mbf V_{d}(\ul{X}) {\equiv\mbf{V}_{d}}$ the column vector of words of degree $d$. We also denote by $\mbf W_{d}$ (resp. $\mbf V_{d}$) the set of all entries of $\mbf W_{d}(\ul{X})$ (resp. $\mbf V_{d}(\ul{X})$).
The length of $\mbf W_{d}$ is equal to  $\s(d,n) := \sum_{i=0}^{d} n^i$, which we write as $\s(d)$, when contextually appropriate. 
Given a polynomial $f\in\mbb{R}\lr{\ul{X}}_{d}$, let $\mbf f = (f_{w})_{w\in \mbf{W}_{d}}\in\mbb{R}^{\s(d)}$ be its vector of coefficients. 
It is clear that every polynomial $f\in\mbb{R}\lr{\ul{X}}_{d}$ is of the form $f = \sum_{w\in \mbf{W}_{d}}f_{w} w = \mbf{f}^{*}\mbf{W}_{d} = \mbf{W}_{d}^{*}\mbf{f}$. For $f\in\mbb{R}\lr{\ul{X}}$ let $\ceil{f} = \ceil{\deg(f)/2}$, and given some $k\in\mbb{N}$, we define $k_{f} : = k - \ceil{f}$, e.g., $\mbf{W}_{k-\ceil{f}} = \mbf{W}_{k_{f}}$. 
We use standard notations on $\mbb{R}^m$, i.e., given $\mbf{a}\in\mbb{R}^m$, $\norm{\mbf{a}}_{2}$ denotes the usual $2$-norm of $\mbf{a}$.

Let $\mbb{S}^{r}$ denote the space of real symmetric matrices of size $r$, we will normally omit the subscript $r$ when we discuss matrices of arbitrary size, or if the size is clear from context. 
Given $\mbf{A}\in\mbb{S}$, $\mbf{A}$ is positive semidefinite (psd) (resp. positive definite (pd)), if all eigenvalues of $\mbf{A}$ are non-negative (resp. positive), and we write $\mbf{A}\succeq0$ (resp. $\mbf{A}\succ0$). 
We denote by $\tr(\mbf{A})$ the trace {($\sum_{i=1}^{r} A_{i,i}$)} of the matrix $\mbf{A} \in \mbb{S}^{r}$ and $\normtr(\mbf{A}) = \frac{1}{r} \tr(\mbf{A})$ {is the} 
normalized trace.
Let $\mbb{S}_{+}$ (resp. $\mbb{S}_{++}$) be the cone of psd (resp. pd) matrices.
For a subset $\mc{S}\subseteq\mbb{S}$, we define $\mc{S}_{+} := {\mc{S}\cap\mbb{S}_{+}}$ 
and $\mc{S}_{++}:={\mc{S}\cap\mbb{S}_{++}}$. 
We write $\ul{A} = (\mbf{A}_{1}, \dotsc, \mbf{A}_{n})\in\mbb{S}^{n}$, and given $q\in\mbb{R}\lr{\ul{X}}$, by $q(\ul{A})$ we mean the evaluation of $q(\ul{X})$ on $\ul{A}$, i.e., replacement of the nc letters $X_{i}$ with the matrices $\mbf{A}_{i}$. We write $\diag(\mbf{B}_{1},\dotsc,\mbf{B}_{r})$ for the block diagonal matrix with diagonal blocks being $\mbf{B}_{i}$.

Finally, given a positive $m\in\mbb{N}$, we write $\mbb{N}^{\geq m} = \set{ m, m+1, \dotso }$, $[m] = \set{1,\dotsc,m}$, and we use $\abs{\cdot}$ to denote the cardinality of a set.

\subsection{Algebraic and geometric structures}\label{SS-AG}

Let $\mf{g} = \set{g_{0},\dotsc,g_{m}}$ and $\mf{h} = \set{h_{1}, \dotsc, h_{\ell}}$ be subsets of  $\Sym\mbb{R}\lr{\ul{X}}$, with the requirement that $g_{0}=1$, unless otherwise stated.

\subsubsection{Quadratic modules}
The \emph{quadratic module} generated by $\mf{g}$ is the set
$$
Q(\mf{g}) := \set{ \sum_{i=0}^m \sum_{j=1}^{r_j} p_i^{(j)*} g_i p_i^{(j)}  \ :\  r_j \in\N^{\ge 1}\,,\,p_i^{(j)}\in \R\langle\underline{ X}\rangle }\,.
$$
The \emph{ideal} generated by the set $\mf{h}$ is the set 
$
I(\mf{h}) := Q(\{\pm h_1,\dots,\pm h_\ell\})
$.
The quadratic module associated to $\mf{g}=\{g_0\}$, is the set of \emph{sums of Hermitian squares} (SOHS). 

 Given $k\in\mbb{N}$, the $k^{\text{th}}$-order truncation of $Q(\mf{g})$ (resp. $I(\mf{h})$), denoted by $Q_{k}(\mf{g})$ (resp. $I_{k}(\mf{h})$), is the set of all polynomials in $Q(\mf{g})$ (resp. $I(\mf{h})$) with degree at most $2k$.
Moreover, one has
\begin{equation*}
	\begin{gathered}
		Q_{k}(\mf{g}) = \set{ \sum_{i=0}^{m} \tr(\mbf{G}_{i}\mbf{W}_{k_{g_{i}}}g_{i} \mbf{W}_{k_{g_{i}}}^{*} ) : \mbf{G}_{i}\succeq0}\,,\\
		I_{k}(\mf{h})= \set{ \sum_{i=1}^{\ell} \tr(\mbf{H}_{i} \mbf{W}_{k_{h_i}} h_{i} \mbf{W}_{k_{h_i}}^{*}) : \mbf{H}_{i} \in\mbb{S} }\,.
	\end{gathered}
\end{equation*}
We say that $Q(\mf{g})+I(\mf{h})$ is \emph{Archimedean} if for all $q\in\mbb{R}\lr{\ul{X}}$, there is a positive $R\in\mbb{N}$ such that $R-q^*q\in Q(\mf{g})+I(\mf{h})$.

\subsubsection{Semialgebraic sets}

We define the semialgebraic set associated to $\mf{g}$ as
$$
\mc{D}_{\mf{g}} = \set{ \ul{A}\in\mbb{S}^{n} : \forall g\in \mf{g}, \, g(\ul{A})\succeq0 }\,.
$$
We can naturally extend this notion from matrix tuples of the same order, to bounded self-adjoint operators on some Hilbert space $\mc{H}$, which make $g(\ul{A})$ psd for all $g\in \mf{g}$. 
This extension is called the \emph{operator semialgebraic set} associated to  $\mf{g}$, and we denote it as $\mc{D}_{\mf{g}}^{\infty}$.

Similarly we define the \emph{variety} associated to $\mf{h}$ as
$$
\mc{D}_{\mf{h}} = \set{ \ul{A}\in\mbb{S}^{n} : \forall h\in \mf{h}, \, h(\ul{A})=0 }\,,
$$
and the natural extension to the \emph{operator variety} $\mc{D}_{\mf{h}}^{\infty}$.

\subsubsection{Hankel matrices and the Riesz functional}

Suppose we have a truncated real valued sequence $\mbf{y} = (y_{w})_{w\in \mbf{W}_{2d}}$. For each such sequence, we define the \emph{Riesz functional}, $L_{\mbf{y}}:\mbb{R}\lr{\ul{X}}_{2d}\rightarrow\mbb{R}$ as
$L_{\mbf{y}}(q) := \sum_{w} q_{w}y_{w}
$ for $q=\sum_{w} q_{w} w\in \mbb{R}\lr{\ul{X}}_{2d}$.

Suppose further that $\mbf{y}$ satisfies
$
y_{w} = y_{w^{*}} $ for all $ w\in \mbf{W}_{2d}
$.
We associate to such $\mbf{y}$ the \emph{nc Hankel matrix} of order $d$, $\mbf{M}_{d}(\mbf{y})$, defined as
$
(\mbf{M}_{d}(\mbf{y}))_{u,v} = L_{\mbf{y}}(u^{*}v)
$,
where $u,v\in \mbf{W}_{d}$.
Given $q\in\Sym\mbb{R}\lr{\ul{X}}$, we define the \emph{localizing matrix} $\mbf{M}_{d_{q}}(q\mbf{y})$ as
$
(\mbf{M}_{d_{q}}(q\mbf{y}))_{u,v} = L_{\mbf{y}}(u^{*}qv)
$,
where now $u,v\in \mbf{W}_{d_{q}}$.

\subsection{Eigenvalue minimization}
\label{sec:app.hierarchy}

Given $f \in \Sym\mbb{R}\lr{\ul{X}}$, $\mf{g}, \mf{h}  \subset \Sym\mbb{R}\lr{\ul{X}}$, 
the minimal eigenvalue of $f$ over $\mc{D}_{\mf{g}}^{\infty}\cap \mc{D}_{\mf{h}}^{\infty}$ is given by:
\begin{equation}\label{eq:def.EG}
	\lambda_{\min}(f,\mf{g},\mf{h}) = \inf \set{ \mbf{v}^{*}f(\ul{A})\mbf{v} : \ul{A}\in \mc{D}_{\mf{g}}^{\infty}\cap \mc{D}_{\mf{h}}^{\infty}, \, \norm{\mbf{v}}_2=1 }.
\end{equation}
We will assume that {\emph{the eigenvalue minimization problem} (EG)} 
\eqref{eq:def.EG} has at least one global minimizer. 
We can approximate the solution of EG 
\eqref{eq:def.EG} from below with a hierarchy of converging SOHS relaxations \cite{pironio2010convergent},  indexed by $k\in\mbb{N}$:
\begin{equation*}\label{eq:sos.hierarchy.0}
	\rho_{k}:=\sup\set{\xi\in\mbb{R} : f-\xi\in Q_{k}(\mf{g})+I_{k}(\mf{h})}\,.
\end{equation*}
Each relaxation gives rise to the following SDP
\begin{equation}\label{eq:sos.hierarchy}
	\rho_{k} = \sup \limits_{\xi,\mbf{G}_{i},\mbf{H}_{j}} \set{ \xi\ \left|
	\begin{aligned}
		f-\xi&=\sum_{i=0}^{m} \tr \left( \mbf{G}_{i}  \mbf{W}_{k_{g_{i}}}  g_{i} \mbf{W}_{k_{g_{i}}}^{*} \right)\\
		&\phantom{=}+ \sum_{j=1}^{\ell}  \tr \left( \mbf{H}_{j}\mbf{W}_{k_{h_{j}}} h_j\mbf{W}_{k_{h_{j}}}^{*} \right),\\
		\mbf{H}_{j}&\in\mbb{S}, \text{ and }  \mbf{G}_{i}\succeq 0
	\end{aligned}\right. }\,.
\end{equation}

Our primary interest is in the dual formulation of this SDP, which can be stated as
\begin{equation}\label{eq:moment.hierarchy}
	\tau_{k}:=\inf \limits_{\mbf{y} \in {\mbb{R}^{\s({2k})} }} 
	\set{ L_{\mbf{y}}(f) 
		\left|
		\begin{aligned}
			& y_{1}=1, \mbf{M}_{k}(\mbf{y}) \succeq 0,  \\
			&\mbf{M}_{k_{g_{i}}  }(g_{i}\mbf{y})   \succeq 0, i\in[m],\\
			&\mbf{M}_{k_{h_{j}} }(h_{j} \mbf{y})   = 0 , j\in[\ell]
		\end{aligned}
		\right.
	}\,.
\end{equation}
Let 
$$k_{\min}:=\max\{\ceil{f},\ceil{g_i},\ceil{h_j}:i\in[m],j\in[\ell]\}\,.$$
When $Q(\mf{g})+I(\mf{h})$ is Archimedean, both $(\rho_{k})_{k\in\N^{\geq k_{\min}}}$ and $(\tau_{k})_{k\in\N^{\geq k_{\min}}}$ converge to $\lambda_{\min}(f,\mf{g},\mf{h})$ due to {an} nc analog of Putinar's Positivstellensatz \cite{helton2}.

\subsection{Trace minimization}

Let $f, \mf{g}, \mf{h}$ be as above. The minimal trace of $f$ over $\mc{D}_{\mf{g}}\cap\mc{D}_{\mf{h}}$ is 
\begin{equation}\label{eq:trmin}
    \normtr_{\min}(f, \mf{g}, \mf{h}) = \inf \set{ \normtr(f(\ul{A})) : \ul{A}\in \mc{D}_{\mf{g}}\cap\mc{D}_{\mf{h}} } \,.
\end{equation}
For trace optimization, we need some additional definitions that capture the specific properties of the $\normtr$ operator. 

Let us start first, with \emph{cyclic equivalence}. Given two polynomials $p,q\in\mbb{R}\lr{\ul{X}}$, we say that $p$ is cyclically equivalent to $q$ if $p-q$ is a sum of commutators, i.e., $p-q = \sum_{i=1}^{k}(u_{i}v_{i}-v_{i}u_{i})$ for some $k\in\mbb{N}$ and $u_{i}, v_{i}\in\mbb{R}\lr{\ul{X}}$, and we write $\displaystyle{p\overset{\cyc}{\sim} q}$. One can now define the \emph{cyclic quadratic module} $Q^{\cyc}(\mf{g})$, as the set of all polynomials $f\in\Sym\mbb{R}\lr{\ul{X}}$ which are cyclically equivalent to some {element of} 
$Q(\mf{g})$ (see \cite[Definition 1.56]{burgdorf2016optimization}). 


We cannot in general work with the sets $\mc{D}_{\mf{g}}^{\infty}, \mc{D}_{\mf{h}}^{\infty}$, since the algebra of bounded operators over a Hilbert space $\mc{H}$ does not admit a trace if $\mc{H}$ is infinite dimensional. Instead we restrict to a certain subset of finite von Neumann algebras of type $\I$ and $\II$, a subset of the algebra of bounded operators on $\mc{H}$, and we denote this by $\mc{D}_{\mf{g}}^{\II_{1}}$. 
Then we consider the following relaxation of \eqref{eq:trmin}:
\begin{equation}\label{eq:trmintwoone}
    \normtr_{\min}^{\II_1}(f, \mf{g}, \mf{h}) = \inf \set{ \normtr(f(\ul{A})) : \ul{A}\in \mc{D}_{\mf{g}}^{\II_{1}} \cap \mc{D}_{\mf{h}}^{\II_{1}} } \,.
\end{equation}
A discussion of von Neumann algebras is beyond the scope of this article, and we refer the reader to \cite[Definition 1.59]{burgdorf2016optimization} for more details. 
%
%
An SOHS relaxation hierarchy, indexed by $k\in\mbb{N}$, for \eqref{eq:trmintwoone} can be written as
\begin{equation}\label{eq:tr.SOS.hierarchy}
	\rho^{\normtr}_{k}:=\sup\set{ \xi\in\mbb{R} : f-\xi\in Q^{\cyc}_{k}(\mf{g})+I^{\cyc}_{k}(\mf{h}) }
\end{equation}
which once again, can be written and solved as an SDP. The dual formulation of this SDP, which is our primary interest, is
\begin{equation}\label{eq:tr.SDP.dual}
    \tau^{\normtr}_{k}:=\inf \limits_{\mbf{y} \in {\mbb{R}^{\s({2k})} }} 
	\set{ L_{\mbf{y}}(f) 
		\left|
		\begin{aligned}
			& y_{1}=1, \text{ and }  y_{u}=y_{v} \text{ if } \displaystyle{u\overset{\cyc}{\sim} v},  \\
			& \mbf{M}_{k}(\mbf{y}) \succeq 0,\\
			&\mbf{M}_{k_{g_{i}}  }(g_{i}\mbf{y})   \succeq 0, i\in[m],\\
			&\mbf{M}_{k_{h_{j}} }(h_{j} \mbf{y})   = 0 , j\in[\ell]
		\end{aligned}
		\right.
	}.
\end{equation}
{Compared to the relaxation \eqref{eq:moment.hierarchy} for EG, \eqref{eq:tr.SDP.dual} has several additional linear constraints arising from cyclic equivalences.} 
Under Archimedeanity of $Q(\mf{g})$, $(\rho^{\normtr}_{k})_{k\in\mbb{N}^{\geq k_{\min}}}$ is monotonically increasing, and converges to $\normtr_{\min}^{\II_1}(f,\mf{g},\mf{h})$, see \cite[Corollary {5.5}]{burgdorf2016optimization}. 

%% file: DenseEP.tex
\section{CTP for NC optimization}
\label{sec:ctp.dense.EP}

In this section we develop a framework which exploits {CTP for NCPOPs. Our results below hold for both eigenvalue \eqref{eq:def.EG} and trace \eqref{eq:trmintwoone} minimization hierarchies \eqref{eq:moment.hierarchy} and \eqref{eq:tr.SDP.dual} respectively. } 
We provide sufficient conditions under which CTP is guaranteed, as well as simple linear programming methods to check these conditions. We conclude by examining some special cases.

\subsection{CTP for {Dual Hierarchies}}
\label{sec:spectral.relax}

We first give a precise definition of CTP for {NCPOP}. 
Recall the sets $\mf{g}$ and $\mf{h}$ from \S \ref{sec:defs}. For every $k\in\mbb{N}^{\geq k_{\min}}$, define $\s_{k} :=\sum_{i=0}^{m}\s(k_{g_{i}})$
, and the set $\mc{S}^{(k)}\subseteq\mbb{S}^{\s_{k}}$ as
$$
\mc{S}^{(k)} := \set{ \mbf{Y}\in\mbb{S}^{\s_{k}} : \mbf{Y}=\diag(\mbf{Y}_{0},\dotsc,\mbf{Y}_{m}), \text{ and each } \mbf{Y}_{i}\in\mbb{S}^{\s(k_{g_{i}})}}.
$$

Letting $\mbf{D}_k(\mbf{y}):=\diag(\mbf{M}_{k}(\mbf{y}),\mbf{M}_{k_{g_{1}}}(g_{1}\mbf{y}),\dotsc, \mbf{M}_{k_{g_{m}} }(g_{m}\mbf{y}))$, 
SDP \eqref{eq:moment.hierarchy} can be rewritten {as}
\begin{equation}\label{eq:dual.diag.moment.mat}
	\tau_{k} = \inf\limits_{\mbf{y} \in \mbb{R}^{\s(2k)} }
	\set{  
		L_{\mbf{y}}(f) 
		\left|
		\begin{aligned}
			&y_{1}=1, \mbf{D}_{k}(\mbf{y}) \in \mc{S}^{(k)}_+,\\
			&\mbf{M}_{k_{h_{j}}}(h_{j}\mbf{y})=0,j\in[\ell]
		\end{aligned}
		\right. 
	}
\end{equation}
{ and we can similarly reformulate SDP \eqref{eq:tr.SDP.dual} to 
\begin{equation}\label{eq:tr.SDP.diag}
	\tau^{\normtr}_{k} := \inf\limits_{\mbf{y} \in \mbb{R}^{\s(2k)} }
	\set{  
		L_{\mbf{y}}(f) 
		\left|
		\begin{aligned}
			& y_{1}=1, \text{ and }  y_{u}=y_{v} \text{ if } \displaystyle{u\overset{\cyc}{\sim} v},  \\ &\mbf{D}_{k}(\mbf{y}) \in \mc{S}^{(k)}_+,\\	&\mbf{M}_{k_{h_{j}}}(h_{j}\mbf{y})=0,j\in[\ell]
		\end{aligned}
		\right. 
	}.
\end{equation}
}
\begin{definition}[CTP]\label{def:ctp}
We say that {an NCPOP} 
has CTP if for every $k\in \mbb{N}^{\geq k_{\min}}$, there exists $a_{k}>0$ and $\mbf{P}_{k}\in \mc{S}^{(k)}_{++}$ such that for all $\mbf{y} \in \mbb{R}^{\s(2k)}$,
\begin{equation*}
    \left.
	\begin{array}{rl}
		&\mbf{M}_{k_{h_{j}}}(h_{j}\mbf{y})=0, j\in[\ell],\\
		&y_{1}=1
	\end{array}
	\right\}
	\Rightarrow  
	\tr(\mbf{P}_{k}^{*} \mbf{D}_{k}(\mbf{y}) \mbf{P}_{k})=a_{k}.
\end{equation*}
\end{definition}
In other words, we say that {NCPOP \eqref{eq:def.EG} or \eqref{eq:trmintwoone}} 
has CTP if each {dual SDP} relaxation {\eqref{eq:dual.diag.moment.mat} or \eqref{eq:tr.SDP.diag}} has an equivalent form involving a psd matrix whose trace is constant. In this case, $a_{k}$ is the constant trace and $\mbf{P}_{k}$ is the change of basis matrix. 
The next proposition is an example of an {NCPOP} 
which has CTP.
\begin{proposition}[nc polydis{c} equality]
\label{pro:dense.ctp}
	Let $m=0$, $\ell\geq n$ and $h_{j}=X_{j}^{2}-1$, for $j\in[n]$. 
	Then 
	\begin{equation}
   		\left.
		\begin{array}{rl}
			&\mbf{M}_{ k_{ h_{j} } }(h_{j}\mbf{y})=0, j\in[\ell],\\
			&y_{1}=1
		\end{array}
		\right\}
		\Rightarrow  \tr\left(\mbf{D}_{k}(\mbf{y})\right)=\s(k).
	\end{equation}
\end{proposition}

\begin{proof}
	Note that $\mbf{D}_{k}(\mbf{y})=\mbf{M}_{k}(\mbf{y})$  since $\mf{g}=\{1\}$.
	Suppose that $\mbf{M}_{k_{h_{j}}}(h_{j}\mbf{y})=0$, {$j\in[\ell]$,} 
	and $y_1=1$. This implies that for every $j\in[n]$, the diagonal of $\mbf{M}_{k_{h_{j}}}(h_{j}\mbf{y})$ is zeros, i.e., $L_{\mbf{y}}(u^{*}(X_{j}^{2}-1)u)=0$, for all $u\in \mbf{W}_{k-1}$. This now implies, for every $w=X_{i_1}\dots X_{i_r}\in \mbf{W}_{k}$
	\begin{equation*}
		\begin{aligned}
			y_{w^{*}w} & = L_{\mbf{y}}(w^{*}w)= L_{\mbf{y}}(X_{i_{r}}\dotso X_{i_{1}} X_{i_{1}}\dotso X_{i_{r}})\\
			&=L_{\mbf{y}}(X_{i_{r}}\dotso X_{i_{2}}(X_{i_{1}}^{2}-1)X_{i_{2}}\dotso X_{i_{r}})\\
			&\phantom{==}+L_{\mbf{y}}(X_{i_{r}}\dotso X_{i_{2}}X_{i_{2}}\dotso X_{i_{r}})\\
			& = L_{\mbf{y}}(X_{i_{r}}\dotso X_{i_{2}}X_{i_{2}}\dotso X_{i_{r}})\\
			& = \dotsb = L_{\mbf{y}}(X_{i_{r}}X_{i_{r}})= 	L_{\mbf{y}}(X_{i_{r}}^{2}-1)+L_{\mbf{y}}(1)=y_{1}=1.
		\end{aligned}
	\end{equation*}	
	This yields  $\tr(\mbf{M}_{k}(\mbf{y}))=\sum_{w\in \mbf{W}_{k}} y_{w^{*}w}= \s(k)$.
\end{proof}

A general solution method for solving {NCPOPs} 
which satisfy CTP can be described as follows. We first convert the $k$-th order relaxation \eqref{eq:dual.diag.moment.mat} {or \eqref{eq:tr.SDP.diag}} to a standard (primal) SDP with CTP and then leverage appropriate first-order algorithms, such as CGAL \cite{yurtsever2019conditional} or \emph{spectral method} (SM) \cite[Appendix A.3]{mai2020exploiting}, which exploit CTP to solve the SDP.

For a detailed exposition on how the SDP \eqref{eq:dual.diag.moment.mat} {or \eqref{eq:tr.SDP.diag} can be converted} to a standard (primal) form, the reader is invited to consult \cite{mai2020hierarchy}. There one will also find explanations of how the primal and dual forms of the SDP are related, and their use with CGAL/SM.

\subsection{Sufficient condition to have CTP}
\label{sec:sufficient.CTP}
We now provide a sufficient condition for {NCPOP} 
to satisfy CTP.
For $k\in\mbb{N}^{\geq k_{\min}}$, let $Q_{k}^\circ(\mf{g})$ be the interior of the truncated quadratic module $Q_{k}(\mf{g})$, i.e.,
\begin{equation*}
    Q_{k}^\circ(\mf{g}):=\set{\sum_{i=0}^{m}  \tr(\mbf{G}_{i}\mbf{W}_{k_{g_{i}}} g_{i}\mbf{W}_{k_{g_{i}}}^{*}) : \mbf{G}_{i}\succ 0 }.
\end{equation*}
\begin{theorem}\label{theo:suff.cond.CTP}
Suppose that for every $k\in\mbb{N}$, the following inclusion holds:
\begin{equation}\label{eq:suffi.con.ideal}
	\mbb{R}^{>0} \subset Q_{k}^\circ(\mf{g}).
\end{equation}
Then {NCPOP \eqref{eq:def.EG} and \eqref{eq:trmin}} 
satisfy CTP.
\end{theorem}

\begin{proof}
	Let $k\in\mbb{N}^{\ge k_{\min}}$ and $a_k>0$ such that $a_k\in  Q_{k}^\circ(\mf{g})$. Then we can write
	\if{
	\begin{equation}\label{eq:rep}
		a_k=\sum_{i=0}^m  \tr\left(\mbf{G}_i\mbf{W}_{k_{g_i}}g_{i}  \mbf{W}_{k_{g_i}}^{*}\right)+\sum_{j=1}^{\ell}\tr\left( \mbf{H}_j\mbf{W}_{k_{h_j}}h_{j} \mbf{W}_{k_{h_j}}^{*}\right),
	\end{equation}
	}\fi
		\begin{equation}\label{eq:rep}
		a_k=\sum_{i=0}^m  \tr\left(\mbf{G}_i\mbf{W}_{k_{g_i}}g_{i}  \mbf{W}_{k_{g_i}}^{*}\right),
	\end{equation}
	with each $\mbf{G}_{i}\in\mbb{S}^{++}$. 
	We denote by $\mbf{G}_{i}^{1/2}$
	the square root of $\mbf{G}_{i}$.
	Set $\mbf{P}_{k}=\diag(\mbf{G}_{0}^{1/2},\dotso,\mbf{G}_{m}^{1/2})$.
	\if{
	Let $\mbf{y}\in\mbb{R}^{\s(2k)}$ such that $\mbf{M}_{k_{h_{j}}}(h_{j}\mbf{y})=0$, for $j\in[\ell]$, and $ y_{1}=1$.
	Then 
	\begin{gather*}
		\begin{aligned}
	L_{\mbf{y}}\left(\sum_{j=1}^{\ell}\tr \left( \mbf{H}_{j} \mbf{W}_{k_{h_{j}}} h_{j} \mbf {W}_{k_{h_j}}^{*} \right)\right)
			=\sum_{j=1}^{\ell}\tr \left( \mbf{H}_{j}\mbf{M}_{k_{h_j}}(\mbf{y})\right)=0.	
		\end{aligned}
	\end{gather*}
	}\fi
	From this and \eqref{eq:rep},
	\begin{equation*}
	\begin{array}{rl}
			a_k& =L_{\mbf{y}} \left( \sum_{i=0}^{m}  \tr \left( \mbf{G}_{i}\mbf{W}_{k_{g_{i}}} g_{i}  \mbf{W}_{k_{g_{i}}}^{*} \right)\right) =\sum_{i=0}^{m}\tr \left(\mbf{G}_{i}\mbf{M}_{k_{g_{i}}} (g_{i}\mbf{y}) \right)\\
			& =\sum_{i=0}^{m} \tr \left( \mbf{G}_{i}^{1/2}\mbf{M}_{k_{g_{i}}}(g_{i}\mbf{y})\mbf{G}_{i}^{1/2} \right) =\tr \left( \mbf{P}_{k} \mbf{D}_{k}(\mbf{y})\mbf{P}_{k}\right).
	\end{array}
	\end{equation*}
\end{proof}

\ext{
\begin{theorem}
Conversely, suppose that NCPOP \eqref{eq:def.EG} or \eqref{eq:trmintwoone} satisfies CTP, and that $\mc{D}^{\circ}_{\mf{g}}\neq \emptyset$, then 
\begin{equation}
\mbb{R}^{>0} \subset Q_{k}^{\circ}(\mf{g})
\end{equation}
\end{theorem}

\begin{proof}
Let $k\geq \mbb{N}^{k_{\min}}$ be fixed, and suppose that NCPOP \eqref{eq:def.EG} or \eqref{eq:trmintwoone} satisfies CTP. Then by definition, there exists $P_{k} = (G_{0}, \dotsc, G_{m}) \in\mc{S}_{++}^{(k)}$ such that for all $\mbf{y}\in\mbb{R}^{\s(2k)}$ one has
\begin{align}\label{eq:poly}
    0<a_{k} & = \tr \left( P_{k}D_{k}P_{k} \right) \nonumber \\
            & = \tr \left( G_{0}^{2} M_{k}(\mbf{y}) \right) + \sum_{i\in [m]}\tr \left( G_{i}^{2} M_{k_{g_{i}}}(g_{i}\mbf{y}) \right) \nonumber \\
            & = L_{\mbf{y}} \left( \mbf{W}_{k} G_{0}^{2} \mbf{W}_{k}^{*} + \sum_{i\in[m]} \mbf{W}_{k_{g_{i}}} G_{i}^{2} \mbf{W}_{k_{g_{i}}}^{*} g_{i} \right) 
\end{align}    
Since this is true for all $\mbf{y}$, it is in particular true if $\mbf{y}$ has a representing measure. So, let $\ul{A}\in\mc{D}_{\mf{g}}^{\circ}$ and let $\theta$ be the uniform measure on an nc $\varepsilon$-neighborhood of $\ul{A}$. Let $\mbf{y}$ be the tracial moment sequence generated by $\theta$. Then, defining the polynomial
$$
q = \mbf{W}_{k} G_{0}^{2} \mbf{W}_{k}^{*} + \sum_{i\in[m]} \mbf{W}_{k_{g_{i}}} G_{i}^{2} \mbf{W}_{k_{g_{i}}}^{*} g_{i} \in Q_{k}^{\circ}(\mf{g})
$$
and continuing from \eqref{eq:poly} we obtain
\begin{align}
    L_{\mbf{y}} \left( \mbf{W}_{k} G_{0}^{2} \mbf{W}_{k}^{*} + \sum_{i\in [m]} \mbf{W}_{k_{g_{i}}} G_{i}^{2} \mbf{W}_{k_{g_{i}}}^{*} g_{i} \right) & = \int_{\mbb{S}} q(\ul{A}) \ d\theta(\ul{A}) \\
    \Rightarrow \int_{\mbb{S}} \left( q(\ul{A})-a_{k} \right) \ d\theta(\ul{A}) & =  0
\end{align}
Since $\mc{D}^{\circ}_{\mf{g}}$ is non-empty, and $q$ is strictly positive, it implies that $q-a_{k}$ vanishes identically on $\mc{D}_{\mf{g}}$, and hence $a_{k} = q\in Q^{\circ}_{k}(\mf{g})$.
\end{proof}
}

The following lemmas will be used later on.

\begin{lemma}\label{lem:sum.equal.len}
For every $r\in{\mbb{N}^{\ge 1}}$, 
there exists a positive real sequence $(c_u^{(r-1)})_{u\in \mbf{W}_{r-1}}$ such that
\begin{equation}\label{eq:ball.decomp}
     \sum_{w\in\mathbf V_r}w^* w  = 1+\sum_{u \in \mbf{W}_{r-1}} c_u^{(r-1)} u^*\left(\sum_{{j\in[n]}}
     X_j^2 - 1\right)u   \,.
\end{equation}
\end{lemma}

\begin{proof}
We intend to prove \eqref{eq:ball.decomp} by induction on $r$.
One has $\sum_{w\in\mathbf V_1}w^* w  = \sum_{{j\in[n]}}
X_j^2=1+(\sum_{{j\in[n]}}
X_j^2-1)$ since $\mathbf V_1=(X_j)_{j\in[n]}$, yielding that \eqref{eq:ball.decomp} is true with $r=1$.
Assume that \eqref{eq:ball.decomp} is true with $r=t$. We claim that \eqref{eq:ball.decomp} is true with $r=t+1$ when we choose
\begin{equation}
    \forall u\in \mathbf W_{t}\,,\quad c_{u}^{(t)}:=\begin{cases}
c_{u}^{(t-1)}+1 & \text{if }u\in \mbf{W}_{t-1}\,,\\
1 & \text{otherwise}\,.
\end{cases}
\end{equation}
Indeed, it holds that
\begin{equation*}
\begin{array}{rl}
     &\sum_{w\in\mathbf V_{t+1}}w^* w \\
     =& \sum_{i_1,\dots,i_{t+1}\in[n]} X_{i_1}\dots X_{i_{t+1}}X_{i_{t+1}}\dots X_{i_1}  \\
     =& \sum_{i_1,\dots,i_{t+1}\in[n]} X_{i_1}\dots X_{i_{t}} (X_{i_{t+1}}^2-1/n) X_{i_t}\dots X_{i_1} \\
     &+  \frac{1}{n}\sum_{i_1,\dots,i_{t+1}\in[n]}  X_{i_1}\dots X_{i_{t}}X_{i_{t}}\dots X_{i_1}\\
     =& \sum_{i_1,\dots,i_{t}\in[n]} X_{i_1}\dots X_{i_{t}} (\sum_{i_{t+1}\in[n]}X_{i_{t+1}}^2-1) X_{i_{t}}\dots X_{i_1} \\
     &+  \sum_{\tilde{w}\in\mathbf W_{t}} \tilde{w}^*\tilde{w}\\
     =& \sum_{v\in \mathbf W_t} v^* (\sum_{j\in[n]}X_{j}^2-1) v\\
     &+1+  \sum_{u \in \mathbf W_{t-1}} c_u^{(t-1)} u^*\left(\sum_{{j\in[n]}}
     X_j^2 - 1\right)u  \,,
\end{array}
\end{equation*}
where the latter equality is due to the induction assumption.
\end{proof}

\begin{lemma}\label{lem:equality} 
For every $k\in{\mbb{N}}^{\ge 1}$, 
there exists a positive real sequence $(d_u^{(k-1)})_{u\in\mathbf W_{k-1}}$ such that
\begin{equation}\label{eq:sum.len.small}
    \sum_{w\in\mathbf W_k}w^* w  = 1+k+ \sum_{u \in\mathbf W_{k-1}} d_u^{(k-1)} u^*\left(\sum_{{j\in[n]}}
    X_j^2 - 1\right)u  \,.
\end{equation}
\end{lemma}
\begin{proof}
Let $k\in{\mbb{N}}$. 
From Lemma \ref{lem:sum.equal.len}, we obtain that
{
\begin{align*}
    \sum_{w\in\mathbf W_k}w^* w & = 1+\sum_{r\in[k]}\sum_{w\in\mathbf V_r}w^* w\\
    & = 1+k+\sum_{r\in[k]}\sum_{u \in\mathbf W_{r-1}} c_u^{(r-1)} u^*\left(\sum_{{j\in[n]}}
    X_j^2 - 1\right)u\,,
\end{align*}
}
yielding the selection 
$d_u^{(k-1)}=\sum_{r\in[\deg(u)+1]} c_u^{(r-1)}$, for $u\in\mathbf{W}_{k-1}$ in \eqref{eq:sum.len.small}.
Hence the desired result follows.
\end{proof}

{ The next result shows that CTP is satisfied whenever an NCPOP involves  a ball constraint.} 
For a real symmetric matrix $\mathbf A$, denote the largest eigenvalue of $\mathbf A$ by $\lambda_{\max}(\mathbf A)$.
\begin{theorem}\label{theo:generic.ctp}
If $1-\sum_{j\in [n]} X_j^2\in \mf{g}$ then the inclusions \eqref{eq:suffi.con.ideal} hold and therefore {NCPOP \eqref{eq:def.EG} and \eqref{eq:trmin}} 
have CTP.
\end{theorem}

\begin{proof}
Without loss of generality, set $g_m:=1-\sum_{j\in[n]} X_j^2$ and let $k\in\N^{\ge k_{\min}}$ be fixed. 
By Lemma \ref{lem:equality}, 
$$a_k=\tr(\mathbf W_k \mathbf W_k^*)+\tr(\mathbf G_m\mathbf W_{k-1} g_m \mathbf W_{k-1}^*)\,,$$
where  $a_k=1+k$ and $\mathbf G_m=\diag((d_u^{(k-1)})_{w\in\mathbf W_{k-1}})$ is pd.
Denote by $\mathbf I_t$ the identity matrix of size {$\s(t)$} 
for $t\in\N$.
Let $\mathbf U$ be a real symmetric matrix such that
\begin{equation*}
    \sum_{i=1}^{m-1}  \tr(\mathbf W_{{k_{g_{i}}}
    } g_i\mathbf W_{{k_{g_{i}}}
    }^*) = \tr(\mathbf U \mathbf W_{k}\mathbf W_{k}^*)\,.
\end{equation*}
Let $\delta>0$ such that $\mathbf I_k-\delta \mathbf U\succ 0$, namely, $\delta=1/(|\lambda_{\max}(\mathbf U)|+1)$.
Note $\mathbf G_0:=\mathbf I_k-\delta \mathbf U$.
Then
\begin{equation*}
\begin{array}{rl}
    a_k= & \tr(\mathbf G_0\mathbf W_k \mathbf W_k^*)+\delta\sum_{i=1}^{m-1} \tr( \mathbf W_{{k_{g_{i}}}
    } g_i \mathbf W_{{k_{g_{i}}}
    }^*) \\
     & +\tr(\mathbf G_m\mathbf W_{k-1} g_m \mathbf W_{k-1}^*)\,,
\end{array}
\end{equation*}
which implies $a_k\in Q^\circ_k(\mf{g})$, the desired result.
\end{proof}
Even though this is not of crucial interest in the context of this paper, we mention that Theorem \ref{theo:generic.ctp} can be used to prove that strong duality holds for the primal-dual \eqref{eq:sos.hierarchy}-\eqref{eq:moment.hierarchy} for all $k\ge k_{\min}$ (see also \cite[Theorem  3.6]{wang2020exploiting} which is an nc analog of \cite{josz2016strong}).
\if{
The next result follows from Theorem \ref{theo:generic.ctp}.  
If an EG \eqref{eq:def.EG} has an nc ball constraint then SDP  \eqref{eq:sos.hierarchy}  has a strictly feasible solution, and thus satisfies Slater's condition.

{I think either the statement of the corollary below should change to $\mf{g}=\set{ 1, 1-\sum_{j=1}^{n}X_{j}^{2} }$, or in the proof we change to $\lambda_{\min}(f,\mf{g},\emptyset)$ with the additional comment that Proposition 4.17 of \cite{burgdorf2016optimization} works on any subset of the ball.
}

\begin{corollary}\label{coro:stric.fea.sol.sos}
Assume that $1-\sum_{j=1}^n X_j^2\in \mf{g}$.
Then Slater's condition holds for SDP \eqref{eq:sos.hierarchy} for all $k\ge k_{\min}$. As a consequence, strong duality holds for the primal-dual \eqref{eq:sos.hierarchy}-\eqref{eq:moment.hierarchy} for all $k\ge k_{\min}$.
\end{corollary}
\begin{proof}
It suffices to prove that SDP \eqref{eq:sos.hierarchy} has a strictly feasible solution for all $k\ge k_{\min}$. 
Let $k\ge k_{\min}$ be fixed. 
From \cite[Proposition 4.17]{burgdorf2016optimization}, one has $f-\lambda_{\min}(f,\{g_0,g_m\},\emptyset)\in Q_k(\{g_0,g_m\})$.
By Theorem \ref{theo:generic.ctp}, $1\in Q_k^\circ(\mf g)$ and therefore
$f+1-\lambda_{\min}(f,\{g_0,g_m\},\emptyset)\in Q^\circ_k(\mf g)$, which yields the desired conclusion.
\end{proof}
}\fi
The following corollary states that  polynomials positive definite on a semialgebraic set belong to the interior of the truncated quadratic module {for a sufficiently large truncation order} 
when an nc ball constraint is present.

\begin{corollary}
Assume that $Q(\mf{g})$ Archimedean.
Let $q\in\Sym\mbb{R}\lr{\ul{X}}$, such that $q(\ul{A})\succ 0$ for all $\ul{A}\in\mc{D}_{\mf{g}}$. 
If $1-\sum_{j\in[n]} X_j^2\in \mf{g}$, then $q\in Q_{k}^\circ(\mf{g})$ for $k$ sufficiently large.
\end{corollary}

\begin{proof}
Let $1-\sum_{j\in[n]} X_j^2\in \mf{g}$. 
Then for all $\ul{A}=(\mbf{A}_1,\dots,\mbf{A}_n)\in\mc{D}_{\mf{g}}$, $\mbf{I}-\sum_{j\in[n]}\mbf{A}_j^2\succeq 0$, so $\mbf{I}\succeq \mbf{A}_j^2$, $j\in[n]$, where $\mbf{I}$ is the identity matrix.
It implies that $\mc{D}_{\mf{g}}$ is bounded.
Thus there exists a small enough $\varepsilon>0$ such that $(q-\varepsilon)(\ul{A})=q(\ul{A})-\varepsilon \mbf{I}\succ 0$ for all $\ul{A}\in\mc{D}_{\mf{g}}$.
By using the nc analog of Putinar’s Positivstellensatz \cite[Theorem 1.32]{burgdorf2016optimization}, there exists $\tilde k\in\N$ such that $q-\varepsilon\in Q_k(\mf{g})$ for all $k\ge \tilde k$.
Let $k\in\N^{\ge \tilde k}$ be fixed.
By Theorem \ref{theo:generic.ctp}, $\varepsilon\in Q_k^\circ(\mf g)$ and therefore
$q=(q-\varepsilon)+\varepsilon\in Q^\circ_k(\mf g)$, which yields the desired conclusion.
\end{proof}

\begin{remark}
	Combining the proof of Theorem \ref{theo:generic.ctp} with the proof of Theorem \ref{theo:suff.cond.CTP}, one can obtain explicit expressions for $a_{k}$ and $\mbf{P}_{k}$ in Definition \ref{def:ctp}. Namely,
	$a_k=1+k$ and
	$$	\mbf{P}_{k}=\diag\left(\mbf{G}_{0}^{1/2},\sqrt{\delta}\mbf{I}_{k_{g_{1}}},\dotso,\sqrt{\delta}\mbf{I}_{k_{g_{m-1}}},\mbf{G}_{m}^{1/2}\right).
	$$
	However, {in our experience} this choice {leads to} 
	poor numerical properties. 
	In the next section we provide a hierarchy of linear programs (LPs) inspired from the inclusions \eqref{eq:suffi.con.ideal}, to obtain the constant trace $a_{k}$ and the change of basis matrix $\mbf{P}_{k}$ which achieve a better numerical performance.
\end{remark}

\subsection{Verifying CTP via linear programming}\label{sec:obtain.CTP.SDP}

For any $k\in\mbb{N}^{\geq k_{\min}}$, let $\widehat{\mc{S}}^{(k)}$ be the set of real diagonal matrices of size $\s(k)$ and consider the following linear program (LP)
\begin{equation}\label{eq:find.CTP.diag}
	\inf \limits_{\xi,\mbf{G}_{i},\mbf{H}_{j}} \set{ \xi\ \left|
		\begin{aligned}
				\xi&=\sum_{i=0}^{m} \tr \left( \mbf{G}_{i}  \mbf{W}_{k_{g_{i}}}  g_{i} \mbf{W}_{k_{g_{i}}}^{*}\right)\\
				&\phantom{=}+ \sum_{j=1}^{\ell}  \tr \left( \mbf{H}_{j}\mbf{W}_{k_{h_{j}}} h_{j}\mbf{W}_{k_{h_{j}}}^{*}\right),\\
				&\mbf{G}_{i}-\mbf{I}_{i} \in \widehat{\mc{S}}^{(k_{g_{i}})}_{+}, i\in\set{0}\cup [m]
		\end{aligned}
	\right. },
\end{equation}
where $\mbf{I}_{i}$ is the identity matrix of size { $s(k_{g_{i}})$} for $i\in\set{0}\cup [m]$.

\begin{lemma}\label{lem:feas.LP}
If LP \eqref{eq:find.CTP.diag} has a feasible solution $(\xi_{k},\mbf{G}_{i,k},\mbf{H}_{j,k})$ for every $k\in\mbb{N}^{\geq k_{\min}}$, then {NCPOP \eqref{eq:def.EG} and \eqref{eq:trmin}} 
have CTP with $a_k=\xi_k$ and $\mbf{P}_{k}=\diag(\mbf{G}_{0,k}^{1/2},\dotso,\mbf{G}_{m,k}^{1/2})$.
\end{lemma}

The proof of Lemma \ref{lem:feas.LP} is similar to that of Theorem \ref{theo:suff.cond.CTP} with $a_k=\xi_{k}$ and $\mbf{G}_{i}=\mbf{G}_{i,k}$, $i\in\set{0}\cup [m]$.

We provide in the following proposition a more detailed description of some feasible solutions to \eqref{eq:find.CTP.diag} in the special cases of the nc polydis{c} 
and the nc ball.

\begin{proposition}\label{prop:suff.cond.feas.ball}
{
Suppose either $\mf{g}=\set{1,  1-\sum_{i\in[n]} X_{i}^{2}}$ or that $\mf{g}=\set{ \frac{1}{n}-X_{i}^{2}: i\in [n] }\cup\set{1}$.
} Then LP \eqref{eq:find.CTP.diag} has a feasible solution for every $k\in\N^{\ge k_{\min}}$, and therefore {NCPOP \eqref{eq:def.EG} and \eqref{eq:trmin}} 
satisfy CTP.
\end{proposition}
\begin{proof}
{It is sufficient in both cases to show that \eqref{eq:find.CTP.diag} has a feasible solution for every $k\in\N^{\ge k_{\min}}$.}

Let $m=1$ and $g_1=1-\sum_{j\in[n]} X_j^2$. 
By Lemma \ref{lem:equality}, $a_k=\tr(\mathbf W_k \mathbf W_k^*)+\tr(\mathbf G_1\mathbf W_{k-1} g_1 \mathbf W_{k-1}^*)$,
where  $\mathbf G_1=\diag((d_u^{(k-1)})_{w\in\mathbf W_{k-1}})$ is pd.
Denote by $\mathbf I_{s(k)}$ the identity matrix of size $s(k)$.
Thus with large enough $r>0$, $(ra_k,(r\mathbf I_{s(k)},r\mathbf G_{1}),\mathbf 0)$ is a feasible solution of \eqref{eq:find.CTP.diag}, for every $k\in\N^{\ge k_{\min}}$.

{On the other hand, if} $m=n$ and $g_j=\frac{1}{n}-X_j^2$, for $j\in[n]${, then
by} Lemma \ref{lem:equality}, 
$a_k=  \tr(\mathbf W_k \mathbf W_k^*)+\tr(\mathbf G\mathbf W_{k-1}(\sum_{j\in[n]} g_j )\mathbf W_{k-1}^*) 
     =\tr(\mathbf W_k \mathbf W_k^*)+\sum_{j\in[n]}\tr(\mathbf G\mathbf W_{k-1} g_j \mathbf W_{k-1}^*)\,,
$
where the diagonal matrix $\mathbf G=\diag((d_u^{(k-1)})_{w\in\mathbf W_{k-1}})$ is pd.
Thus with large enough $r>0$, $(ra_k,(r\mathbf I_{s(k)},r\mathbf G),\mathbf 0)$ is a feasible solution of \eqref{eq:find.CTP.diag}, for every $k\in\N^{\ge k_{\min}}$.
\end{proof}

Since small constant traces are highly desirable for efficiency of first-order algorithms (e.g. CGAL), we search for an optimal solution of LP \eqref{eq:find.CTP.diag} instead of just a feasible solution.

\subsection{Universal algorithm}
Algorithm \ref{alg:sol.nonsmooth.hier.B} below solves EG \eqref{eq:def.EG} 
where CTP can be verified by LP. {A similar algorithm solves NCPOP \eqref{eq:trmin}.}

\begin{algorithm}
    \caption{SpecialEP-CTP}
    \label{alg:sol.nonsmooth.hier.B} 
    \small
    \textbf{Input:} EG \eqref{eq:def.EG} and a relaxation order $k\in\mbb{N}^{\geq k_{\min}}$\\
    \textbf{Output:} The optimal value $\tau_{k}$ of SDP \eqref{eq:dual.diag.moment.mat}
    \begin{algorithmic}[1]
    \State Solve LP \eqref{eq:find.CTP.diag} to obtain an (optimal) solution $(\xi_{k},\mbf{G}_{i,k},\mbf{H}_{j,k})$;
    \State Let $a_{k}=\xi_{k}$ and $\mbf{P}_{k}=\diag(\mbf{G}_{0,k}^{1/2},\dotso,\mbf{G}_{m,k}^{1/2})$;
    \State Compute the optimal value $\tau_{k}$ of SDP \eqref{eq:dual.diag.moment.mat} by 
    running an algorithm based on first-order methods, and which exploits CTP.
    \end{algorithmic}
\end{algorithm}
Two examples of algorithms based on first-order methods and which exploit CTP are CGAL \cite{yurtsever2019conditional} or {SM} 
\cite[Appendix A.3]{mai2020exploiting}.

%% file: Sparsity.tex
\section{CTP with correlative sparsity}
\label{sec:CTP.correlative.EP}

In this section, we show that the CTP can analogously be exploited for {EG \eqref{eq:def.EG}}  
with \emph{sparse} nc polynomials. For brevity, we {focus on EG, and show only} the framework for \emph{correlative sparsity} (CS) \cite{klep2019sparse}, however the {trace minimization setting, as well as the} frameworks for \emph{term sparsity} (TS) as well as \emph{correlative-term sparsity} (CS-TS) \cite{wang2020exploiting} are very similar. We note that the proofs are very similar to those presented in \S \ref{sec:ctp.dense.EP}, and so we omit them for brevity.
	
To begin with, we define CS and present the associated approximation hierarchies for EG \eqref{eq:def.EG}  satisfying CS, which were initially proposed in \cite{waki2006sums,klep2019sparse}. 

\subsection{EPs with CS}\label{sec:sparse.EP}
For $w=X_{i_{1}}\dotso X_{i_{r}}$, let $\var(w):=\set{i_{1},\dotsc,i_{r}}$.
For $I\subseteq [n]$, let $\ul{X}(I):=\set{X_{j}:j\in I }$ and $\mbf{W}^{I}_{d}(\ul{X}):=\{ w\in \mbf{W}_{d}(\ul{X}) : \var(w)\subseteq \ul{X}(I) \}$  with length $\s(d,\abs{I}):= \sum_{i=0}^{d}\abs{I}^{i}$.
Similarly, we note $\mbf{V}^{I}_{d}(\ul{X}):=\set{ w\in \mbf{V}_{d}(\ul{X}) : \var(w)\subseteq \ul{X}(I) }$.
Given $\mbf{y}=(y_{w})_{w\in\mbf{W}_{2d}}$, and $I\subseteq [n]$, the nc Hankel submatrix associated to $I$ of order $d$ is defined as 
$$
\left( \mbf{M}_{d}(\mbf{y},I) \right)_{u,v} := L_{\mbf{y}}(u^*v), \text{ for } u,v\in \mbf{W}^{I}_{d}
$$ 
and for $q\in\mbb{R}\lr{\ul{X}(I)}$, the localizing (sub)matrix is
$$
( \mbf{M}_{d_{q}}(q\mbf{y}, I) )_{u,v} :=  L_{\mbf{y}}(u^*qv) , \text{ for } u,v\in \mbf{W}^{I}_{d_{q}}.
$$

Assume that $\set{I_{j}}_{j\in[p]}$ (with $n_{j}:=\abs{I_{j}}$) are the maximal cliques of (a chordal extension of) the correlative sparsity pattern (csp) graph associated to EG \eqref{eq:def.EG}, as defined in \cite{waki2006sums,klep2019sparse}.
%
Let $\set{J_{j}}_{j\in[p]}$ (resp. $\set{W_{j}}_{j\in[p]}$) be a partition of $[m]$ (resp. $[\ell]$) such that for all $i\in J_{j}$, $g_{i}\in \mbb{R}\lr{\ul{X}(I_{j})}$ (resp. $i\in W_{j}$, $h_{i}\in \mbb{R}\lr{\ul{X}(I_{j})}$), for every $j\in[p]$.
For each  $j\in [p]$, let $m_{j}:=\abs{J_{j}}$, $l_{j}:=\abs{W_{j}}$ and $\mf{g}_{J_{j}}:=\set{g_{i} : i\in J_{j} }$, $\mf{h}_{W_{j}}:=\set{h_{i}: i\in W_{j} }$.
Then $Q(\mf{g}_{J_{j}})$ (resp. $I(\mf{h}_{W_{j}})$) is a quadratic module (resp. an ideal) in $\mbb{R}\lr{\ul{X}(I_{j})}$, for each $j\in[p]$.

For each $k\in \mbb{N}^{\geq k_{\min}}$, consider the hierarchy of sparse SOHS relaxations
\begin{equation}\label{eq:primal.cs.SOS}
	\rho_{k}^{\text{cs}} := \sup \set{ \xi : f-\xi\in\sum_{j\in[p]} \left( Q(\mf{g}_{J_{j}})_{k}+I(\mf{h}_{W_{j}})_{k} \right) }.
\end{equation}
This relaxation can be stated as a primal SDP similar to \eqref{eq:sos.hierarchy}, but we are mostly interested in the dual SDP, which can be stated as follows
\begin{equation}\label{eq:dual.cs.SDP} 
	\tau_{k}^{\text{cs}} := \inf \limits_{\mbf{y} \in {\mbb{R}^{\s({2k})} }} \set{
		L_{\mbf{y}}(f) \left|
		\begin{aligned}
			& y_{1} =1, \mbf{M}_{k}(\mbf{y}, I_{j}) \succeq 0, j\in[p]. \\
			& \mbf{M}_{k_{g_{i}} }(g_{i} \mbf{y}, I_{j})   \succeq 0, i\in J_{j}, j\in[p], \\
			& \mbf{M}_{k_{h_{i}} }(h_{i} \mbf{y}, I_j) = 0 , i\in W_{j}, j\in[p]
		\end{aligned}
	\right. 
	}.
\end{equation}


It is shown in \cite[Corollary 6.6]{klep2019sparse} that the primal-dual SDP pair arising from \eqref{eq:primal.cs.SOS} are guaranteed to converge to the optimal value if there are ball constraints present on each clique of variables.

\subsection{Exploiting CTP with correlative sparsity}
\label{sec:def.CTP.each.clique}
Consider EG \eqref{eq:def.EG} with CS described in \S \ref{sec:sparse.EP}.
Given $j\in[p]$, $k\in\mbb{N}^{\geq k_{\min}}$, and $\mbf{y}\in\mbb{R}^{\s(2k)}$, let 
$$
\mbf{D}_{k}(\mbf{y}, I_{j} ):=\diag( \mbf{M}_{k}(\mbf{y}, I_{j} ),(\mbf{M}_{k_{g_{i}} }(g_{i}\mbf{y}, I_{j} ) )_{ i\in J_{j} } ) ,
$$
with size denoted by $s_{k,j}$.
Then SDP \eqref{eq:dual.cs.SDP} can be rewritten as
\begin{equation}\label{eq:cs-ts.SDP.simple}
	\tau_{k}^{\text{cs}} = \inf \limits_{\mbf{y} \in {\mbb{R}^{\s({2k})} }} 
	\set{ L_{\mathbf y}(f)
		\left|
		\begin{aligned}
		& y_{1} = 1, \mbf{D}_{k}( \mbf{y}, I_{j} ) \succeq 0 , j\in[p] ,   \\
		& \mbf{M}_{k_{h_{i}} }(h_{i} \mbf{y}, I_{j} )   = 0, i\in W_j, j\in[p]
	\end{aligned}
	\right.
	}.
\end{equation}
As in the dense case, let us define 
\begin{align*}
\mc{S}^{(k,j)} := & \bigg\{ \mbf{Y}\in\mbb{S}^{\s_{k,j}} :  \mbf{Y}=\diag(\mbf{Y}_{0}, (\mbf{Y}_{i})_{i \in J_j}  ), \mbf{Y}_{0} \in \mbb{S}^{\s(k,n_j)} \\ 
& \text{ and each } \mbf{Y}_{i}\in \mbb{S}^{\s(k_{g_{i}},n_i)} \bigg\}.  
\end{align*}
We define CTP for EP with CS as follows.
\begin{definition}\label{def:ctp.cs}(CS-CTP)
We say that EG \eqref{eq:def.EG} with CS has CTP if for every $k\in\mbb{N}^{\geq k_{\min}}$ and for every $j\in[p]$, there exists a positive number $a_{k}^{(j)}$ and $\mbf{P}_{k}^{(j)}\in  \mc{S}^{(k,j)}_{++}    $  such that for all $\mbf{y} \in \mbb{R}^{\s(2k)}$,
\begin{equation*}
    \left.
	\begin{array}{rl}
		&\mbf{M}_{k_{h_{i}} }(h_{i}\mbf{y},I_{j})=0 , i\in W_{j},  \\
		& y_{1}=1
	\end{array}
	\right\}
	\Rightarrow  
	\tr \left( \mbf{P}_{k}^{(j)} \mbf{D}_{k}(\mbf{y},I_{j}) \mbf{P}_{k}^{(j)} \right) = a_{k}^{(j)}.
\end{equation*}
\end{definition}

The following result provides a sufficient condition for an EG \eqref{eq:def.EG} with CS to satisfy CTP. 
\begin{theorem}\label{theo:generic.ctp.cs}
Assume that there is an nc ball constraint on each clique of variables, i.e.,
$1-\sum_{i\in I_{j}} X_{i}^{2} \in \mf{g}_{J_{j}}$,  for every $j \in[p]$.
Then one has
$\mbb{R}^{>0}\subseteq Q_{k}^\circ(\mf{g}_{J_{j}})$, for all $k\in\mbb{N}^{\geq k_{\min}}$ and for all $j\in [p]$.
As a consequence, EG \eqref{eq:def.EG} has CTP.
\end{theorem}

A proof of Theorem \ref{theo:generic.ctp.cs} can be obtained in a similar fashion to \S \ref{sec:sufficient.CTP}, by considering each clique of variables. 
\if{
This approach also leads to the following corollary, similar to how Corollary \ref{coro:stric.fea.sol.sos} was obtained. 
\begin{corollary}\label{coro:slater.con.cs}
	Under the assumptions of Theorem \ref{theo:generic.ctp.cs}, the SDP pair arising from the SOHS relaxation \eqref{eq:primal.cs.SOS} exhibit Slater's condition for all $k\in\mbb{N}^{\geq k_{\min}}$.
\end{corollary}
}\fi

\subsection{Verifying CS-CTP via linear programming}
As in the dense case, given an EG \eqref{eq:def.EG} with CS, we can verify if CS-CTP is satisfied via a hierarchy of LPs.
For every $k\in\mbb{N}^{\geq k_{\min}}$ and for every $j\in[p]$, let $\widehat{\mc{S}}^{(k,j)}$ be the set of real diagonal matrices of size $\s(k,n_{j})$ and consider the following LP
\begin{equation}\label{eq:find.CTP.CS}
	\inf \limits_{ \xi, \mbf{G}_{i}, \mbf{H}_{i} } 
	\set{
		\xi \left|
		\begin{aligned}	&\mbf{G}_{i}-\mbf{I}_{i}^{(j)} \in \widehat{\mc{S}}^{( k_{g_{i}}, j)}_{+} , i\in J_{j}\cup \set{0},  \\
			&\xi= \sum_{i\in \set{0} \cup J_{j} } \tr \left( \mbf{G}_{i} \mbf{W}_{ k_{ g_{i} } }^{ I_{j} } g_{i} 
			(\mbf{W}_{ k_{ g_{i} } }^{ I_{j} } )^{*} \right)  \\
			& + \sum_{i\in W_{j} } \tr \left( \mbf{H}_{i}
			\mbf{W}_{k_{ h_{i} } }^{ I_{j} } h_{i}
			( \mbf{W}_{k_{ h_{i} } }^{ I_{j} } )^{*} \right)
		\end{aligned}
		\right. 
	},
\end{equation}
where $\mbf{I}_{i}^{(j)}$ is the identity matrix of size $s(k_{g_{i}}, j)$, for every $i\in\set{0}\cup J_{j}$.

\begin{lemma}\label{lem:feas.LP.cs}
Let EG \eqref{eq:def.EG} with CS be as described in \S \ref{sec:sparse.EP}. If LP \eqref{eq:find.CTP.CS} has a feasible solution $(\xi_{k}^{(j)},\mbf{G}_{i,k}^{(j)},\mbf{H}_{i,k}^{(j)})$, for every $k\in\mbb{N}^{\geq k_{\min}}$ and for every $j\in[p]$, then EG \eqref{eq:def.EG} satisfies CS-CTP with $\mbf{P}_{k}^{(j)}=\diag( (\mbf{G}_{0,k}^{(j)})^{1/2},( (\mbf{G}_{i,k}^{(j)})^{1/2})_{i\in J_{i}} )$ and $a_{k}^{(j)}=\xi_{k}^{(j)}$, for $k\in\mbb{N}^{\geq k_{\min}}$ and for $j\in[p]$.
\end{lemma}


Similar to \S \ref{sec:ctp.dense.EP}, two special cases where CS-CTP can be verified through LP \eqref{eq:find.CTP.CS}, are the nc polydisc, and the nc ball on each clique of variables. 

\begin{proposition}\label{prop:suff.cond.feas.cs}
Let EG \eqref{eq:def.EG} with CS be as described in \S \ref{sec:sparse.EP}. Suppose either of the following holds
\begin{itemize}
    \item Case 1: $\mf{g}_{J_{j}}=\set{ 1-\sum_{i\in J_{j}} X_{i}^{2} }$, $j\in[p]$.
    \item Case 2: $\mf{g}_{J_{j}}=\set{ \frac{1}{\abs{J_{j}}}-X_{i}^{2}: i\in J_{j} }$, $j\in[p]$.
\end{itemize}
Then LP \eqref{eq:find.CTP.CS} has a feasible solution for every $k\in\N^{\ge k_{\min}}$, and therefore EG \eqref{eq:def.EG} satisfies CS-CTP.
\end{proposition}
The proof of the Proposition \ref{prop:suff.cond.feas.cs} is similar to the dense setting.

\subsection{Universal algorithm}
Algorithm \ref{alg:sol.nonsmooth.hier.B.cs} below solves EG \eqref{eq:def.EG} with CS and whose CS-CTP can be verified by solving LP \eqref{eq:find.CTP.CS}.
\begin{algorithm}
    \caption{SpecialEP-CS-CTP}
    \label{alg:sol.nonsmooth.hier.B.cs} 
    \small
    \textbf{Input:} An EG \eqref{eq:def.EG} with CS and a relaxation order $k\in\mbb{N}^{\geq k_{\min}}$ \\
    \textbf{Output:} The optimal value $\tau_{k}^{\text{cs}}$ of SDP \eqref{eq:dual.cs.SDP}
    \begin{algorithmic}[1]
    
    	\For {$j\in[p]$}{} 
    	\State Solve LP
    	\eqref{eq:find.CTP.CS} to obtain an optimal solution $( \xi_{k}^{(j)}, \mbf{G}_{i,k}^{(j)}, \mbf{H}_{j,k}^{(j)} )$;
    	\State Let $a_{k}^{(j)}= \xi_{k}^{(j)}$ and $\mbf{P}_{k}^{(j)}=\diag( (\mbf{G}_{0,k}^{(j)})^{1/2},\dotsc,(\mbf{G}_{m,k}^{(j)})^{1/2} )$;
    	\EndFor
  		\State Compute the optimal value $\tau_{k}^{\text{cs}}$ of SDP \eqref{eq:cs-ts.SDP.simple}  by running an algorithm based on first-order methods and which exploits CTP.
    \end{algorithmic}
\end{algorithm}

%% file: Experiments.tex
\section{Numerical experiments}
\label{sec:benchmark}
In this section we report results of numerical experiments {conducted on the eigenvalue minimization problem \eqref{eq:def.EG}. These results were} obtained by executing Algorithm \ref{alg:sol.nonsmooth.hier.B} and Algorithm \ref{alg:sol.nonsmooth.hier.B.cs}, respectively for dense and sparse randomly generated instances 
of nc quadratically constrained quadratic problems (QCQPs) with CTP.
In the dense case, one computes the first and second order SDP relaxations, namely the optimal values $\tau_1$ and $\tau_2$ of
 SDP \eqref{eq:dual.diag.moment.mat}.
Similarly in the sparse case, one computes the optimal values $\tau_1^{\text{cs}}$ and $\tau_2^{\text{cs}}$ of
 SDP \eqref{eq:cs-ts.SDP.simple}.
The experiments are performed in Julia 1.3.1 with the following software packages:
\begin{itemize}
    \item {\tt NCTSSOS} \cite{wang2020exploiting} is a modeling library for solving Moment-SOS relaxations of sparse EPs based on JuMP (with Mosek 9.1 used as SDP solver).
    \item {\tt Arpack} \cite{lehoucq1998arpack} is used to compute the smallest eigenvalues and the corresponding eigenvectors of real symmetric matrices of (potentially) large size, which is based on the implicitly restarted Arnoldi method  \cite{lehoucq1996deflation}. 
\end{itemize}

Both implementation of Algorithm \ref{alg:sol.nonsmooth.hier.B} and  \ref{alg:sol.nonsmooth.hier.B.cs} are available online:
\begin{center}
    \href{https://github.com/maihoanganh/ctpNCPOP}{{\bf https://github.com/maihoanganh/ctpNCPOP}}.
\end{center}

We use a desktop computer with an Intel(R) Core(TM) i7-8665U CPU @ 1.9GHz $\times$ 8 and 31.2 GB of RAM. 

We use the following notation for the numerical results. 
The number of variables, inequality and equality constraints are denoted by $n$, $m$ and $l$, respectively.
We denote by $k$ the relaxation order used to solve the dense SDP  \eqref{eq:dual.diag.moment.mat} and the sparse SDP \eqref{eq:cs-ts.SDP.simple}.
For NCPOP with CS, let us denote by
$u^{\max}$ the largest size of variable cliques and
$p$ the number of variable cliques.
We note
$\omega$, $s^{\max}$, $\zeta$ and $a^{\max}$  the number of psd blocks,
the largest size of psd blocks,
the number of affine equality constraints and
the largest constant trace of matrices involved in the SDP relaxations, respectively.
Let ``val'' stand for the approximate optimal value of the SDP relaxation with desired accuracy $\varepsilon$ for CGAL, and let ``time'' be the corresponding running time in seconds.
We use ``$-$'' to indicate that the calculation runs out of space.
For all examples tested in this paper, the modeling time for both {\tt NCTSSOS} and {\tt ctpNCPOP} is typically negligible compared to the solving time of Mosek and CGAL. Hence the total running time mainly depends on the solvers and we compare their performances below.

\subsection{Randomly generated dense QCQPs}
\label{sec:experiment.single.ball}
\paragraph{Test problems:} We construct randomly generated dense QCQPs with nc ball and nc polydisk constraints as follows:
\begin{enumerate}
    \item Generate a dense quadratic polynomial objective function $f=\frac{1}{2}\sum_{w\in \mathbf W_2} \bar f_w(w+w^*)\in\Sym\R\lr{\ul{X}}_2$ with coefficients $\bar f_w$ randomly chosen w.r.t. the uniform probability distribution on $(-1,1)$.
        \item Do one of the following two cases:
        \begin{itemize}
            \item nc ball: let $m=1$ and $g_1:=1-\sum_{r\in[n]}X_r^2$;
            \item nc polydisk:  let $m=n$ and $g_i:=\frac{1}{n}-X_i^2$, $i\in[n]$;
        \end{itemize}
    \item Take a random point $\mathbf a$ in $\{x\in\R^n: g_i(x)\ge 0,i\in[m]\}$ w.r.t. the uniform distribution;
    \item For every $j\in[\ell]$, generate a dense quadratic polynomial $h_j=\frac{1}{2}\sum_{w\in \mathbf W_2} \bar h^{(j)}_w(w+w^*)\in\Sym\R\lr{\ul{X}}_2$:
    \begin{enumerate}[(i)]
        \item for each $w\in\mathbf W_2\backslash \{ 1\}$, select a random coefficient $\bar h^{(j)}_w$  in $(-1,1)$ w.r.t. the uniform distribution;
        \item set $\bar h^{(j)}_1:=-\sum_{w\in \mathbf W_2\backslash \{1\}} \bar h^{(j)}_w  w(\mbf{a})$.
    \end{enumerate}
    Then $\mathbf a$ is a feasible solution of EG \eqref{eq:def.EG}.
\end{enumerate}

\begin{table}
    \caption{\small Numerical results for randomly generated dense QCQPs with nc ball constraint}
    \label{tab:QCQP.on.unit.ball}
    {\small\begin{itemize}
        \item EP size: $m=1$, $l=\lceil n/4\rceil $; SDP size: $\omega=2$, $a^{\max}=3$; CGAL accuracy:  $\varepsilon=10^{-4}$.
    \end{itemize}}
\small
\begin{center}
   \begin{tabular}{|c|c|c|c|c|c|c|c|c|}
        \hline
        \multicolumn{2}{|c|}{EP size}&
        &
        \multicolumn{2}{c|}{SDP size} & \multicolumn{2}{c|}{Mosek}&      \multicolumn{2}{c|}{CGAL} \\ \hline
$n$&$l$&$k$ &$s^{\max}$ & $\zeta$&
\multicolumn{1}{c|}{val}& \multicolumn{1}{c|}{time}&
\multicolumn{1}{c|}{val}&
\multicolumn{1}{c|}{time}\\
\hline
\multirow{2}{*}{10}&\multirow{2}{*}{3} &1 & 11 & 5 & -3.2413 & 1 &-3.2411 & 2  \\
\cline{3-9}
&  & 2&111 & 815 & -3.1110 & 28 &-3.1107 & 59  \\
\hline
\multirow{2}{*}{20}&\multirow{2}{*}{5} &1 & 21 & 7 & -3.5534 & 0.03& -3.5525 & 1  \\
\cline{3-9}
&  & 2&421 & 5587 & $-$& $-$& -3.5026 &203\\
\hline
\multirow{2}{*}{30}&\multirow{2}{*}{8} &1& 31 & 10 & -4.6984 & 0.1 & -4.6954&1   \\
\cline{3-9}
&  & 2&931 & 18415 & $-$ & $-$ & -4.6819 & 1392\\
\hline
\end{tabular}    
\end{center}
\end{table}
\begin{table}
    \caption{\small Numerical results for randomly generated dense QCQPs with nc polydisk constraints}
    \label{tab:QCQP.on.box}
    {\small\begin{itemize}
        \item EP size: $m=n$, $l=\lceil n/7\rceil $;  SDP size: $\omega=n+1$, $a^{\max}=3$, CGAL accuracy: $\varepsilon=10^{-4}$.
    \end{itemize}}
\small
\begin{center}
   \begin{tabular}{|c|c|c|c|c|c|c|c|c|}
        \hline
        \multicolumn{2}{|c|}{EP size}&
        &
        \multicolumn{2}{c|}{SDP size} & \multicolumn{2}{c|}{Mosek}&      \multicolumn{2}{c|}{CGAL} \\ \hline
$n$&$l$& $k$ &$s^{\max}$ & $\zeta$&
\multicolumn{1}{c|}{val}& \multicolumn{1}{c|}{time}&
\multicolumn{1}{c|}{val}&
\multicolumn{1}{c|}{time}\\
\hline
\multirow{2}{*}{10}&\multirow{2}{*}{2}& 1&11 &13 &-3.2165 & 0.009&-3.2154 &0.4\\
\cline{3-9}
& &2 &111 &1343 &-3.2039 &26 &-3.2037 &229 \\\hline
\multirow{2}{*}{20}&\multirow{2}{*}{3}& 1&21 &24 &-4.5773 &0.03 &-4.5767 &2\\
\cline{3-9}
& &2 &421 &9514 &$-$ &$-$ & -4.5147& 753 \\\hline
\multirow{2}{*}{30}&\multirow{2}{*}{3}& 1&31 &36 & -5.1182 &0.8 &-5.1172&3\\
\cline{3-9}
& &2 &931 &31311 &$-$ &$-$ & -5.0717&2215 \\\hline
\end{tabular}    
\end{center}
\end{table}

The numerical results are displayed in Tables  \ref{tab:QCQP.on.unit.ball} and \ref{tab:QCQP.on.box}. The results show that CGAL is typically the fastest solver and returns an approximate optimal value which differs from 1\% w.r.t. the one returned by Mosek when $n\le 20$. Mosek runs out of memory when $n\ge 20$ {and $k=2$,} while CGAL still works well.
With our current setting for the CGAL accuracy, the approximate optimal value is correct up to the two first accuracy digits, so we can guarantee that the bound improves from $k=1$ to $k=2$.

\subsection{Randomly generated QCQPs with CS}
\label{sec:experiment.single.ball.sparse}

\paragraph{Test problems:} We construct randomly generated nc QCQPs with CS and ball constraints on each clique of variables as follows:
\begin{enumerate}
    \item Take a positive integer $u$, $p:=\lfloor n/u\rfloor +1$ and let
    \begin{equation}\label{eq:index.vars}
        I_j=\begin{cases}
        [u],&\text{if }j=1\,,\\
        \{u(j-1),\dots,uj\},&\text{if }j\in\{2,\dots,p-1\}\,,\\
        \{u(p-1),\dots,n\},&\text{if }j=p\,;
        \end{cases}
    \end{equation}
    \item Generate a quadratic polynomial objective function $f=\sum_{j\in[p]}f_j$ such that for each $j\in[p]$, $f_j=\sum_{w\in\mbf{W}_2^{I_{j}}} \bar f_w^{(j)} (w+w^*)\in\Sym\R\lr{\ul{X}(I_{j})}_2$, and the coefficients are randomly generated as in the dense setting;
    \item Take $m=p$ and $g_j:=1-\sum_{i\in I_j} X_i^2$, $j\in[m]$.
    \item Take a random point $\mathbf a$ as in the dense setting;
    \item Let $r:=\lfloor l/p\rfloor$ and \begin{equation}\label{eq:assign.eqcons}
        W_j:=\begin{cases} \{(j-1)r+1,\dots,jr\},&\text{if }j\in[p-1]\,,\\
    \{(p-1)r+1,\dots,l\},&\text{if } j=p\,.
    \end{cases}
    \end{equation}
     For every $j\in[p]$ and every $i\in W_j$, generate a quadratic polynomial $h_i=\frac{1}{2}\sum_{w\in\mbf{W}_2^{I_{j}}} \bar h_w^{(i)} (w+w^*)\in\Sym\R\lr{\ul{X}(I_{j})}_2$ as in the dense setting to ensure that $\mathbf a$ is a feasible solution of EG \eqref{eq:def.EG}.
     \if{
    \begin{enumerate}
        \item for each $w\in\mbf{W}_2^{I_{j}}\backslash \{ 1\}$, taking a random coefficient $\bar h_w^{(i)}$ of $h_i$ in $(-1,1)$ w.r.t. the uniform distribution;
        \item setting $\bar h_{1}^{(i)}:=-\sum_{w\in\mbf{W}_2^{I_{j}}\backslash \{1\}} \bar h_w^{(i)} w(\mbf{a})$.
    \end{enumerate}
    Then $\mathbf a$ is a feasible solution of EG \eqref{eq:def.EG}.
        }\fi
\end{enumerate}


\begin{table}
    \caption{\small Numerical results for randomly generated QCQPs with CS and nc ball constraint on each clique of variables}
    \label{tab:single.ball.sparse.QCQP}
    {\small\begin{itemize}
        \item EP size: $n=1000$, $m=p$, $l=143$, $u^{\max}=u+1$; SDP size: $\omega=2p$, $a^{\max}=3$; CGAL accuracy:  $\varepsilon=10^{-3}$.
    \end{itemize}}
\small
\begin{center}
   \begin{tabular}{|c|c|c|c|c|c|c|c|c|c|}
        \hline
        \multicolumn{2}{|c|}{EP size}&
        &
        \multicolumn{3}{c|}{SDP size} & \multicolumn{2}{c|}{Mosek}&      \multicolumn{2}{c|}{CGAL} \\ \hline
$u$&$p$ &$k$ &$\omega$&$s^{\max}$ & $\zeta$&
\multicolumn{1}{c|}{val}& \multicolumn{1}{c|}{time}&
\multicolumn{1}{c|}{val}&
\multicolumn{1}{c|}{time}\\
\hline
\multirow{2}{*}{10}&\multirow{2}{*}{100}& 1&200 &12 &541 & -2.9659 & 3 & -2.9662&206  \\
\cline{3-10}
& & 2&200 &133 &91850 & -2.9594 & 32008 & -2.9598& 7790 \\
\hline
\multirow{2}{*}{15}&\multirow{2}{*}{66}& 1&132 &27 &405 & -2.3230 & 1 &-2.3225 &38 \\
\cline{3-10}
& &2 &132&703&185592&$-$&$-$&-2.3179&10051\\
\hline
\multirow{2}{*}{20}&\multirow{2}{*}{50}& 1&100 &22 &341 & -2.1517 & 4 &-2.1515& 54 \\
\cline{3-10}
&&2&100&463&290908&$-$ &$-$& -2.1260& 11791 \\
\hline
\end{tabular}    
\end{center}
\end{table}
The number of variables is fixed as $n=1000$. 
We increase the clique size $u$ so that the number of variable cliques $p$ decreases accordingly.
The numerical results are displayed in Table  \ref{tab:single.ball.sparse.QCQP}. 
Again results in Table  \ref{tab:single.ball.sparse.QCQP} show that CGAL is slower than Mosek for $k=1$ but is  faster for $k=2$ and returns an approximate optimal value which differs from 1\% w.r.t. the one returned by Mosek (for $u\le 10$). Mosek runs out of memory for $k=2$ when $u\ge 15$, while {CGAL is once again able to obtain improved lower bounds.}

\balance




%% file: Conclusion.tex
\section{Conclusion}
We have provided a constructive proof that the constant trace property holds for semidefinite relaxations of eigenvalue or trace optimization problems, whenever an nc ball constraint is present. 
This property can be easily verified by solving a hierarchy of linear programs, when the only involved inequality constraints are either noncommutative ball or nc polydisk constraints.
This allows one to use first order methods exploiting the constant trace property (e.g., CGAL) to solve the semidefinite relaxations of large-scale eigenvalue problems more efficiently than with second order interior-point solvers (e.g., Mosek). We have experimentally demonstrated some of these computational gains on eigenvalue minimization. Similar gains shall be achievable for trace minimization.
For many testing examples in
this paper, the relative optimality gap of CGAL w.r.t. Mosek is always smaller than 1\%.

As a topic of further research, we intend to rely on our framework to tackle applications arising from quantum information and condensed matter, including bounds on maximal violation levels for Bell inequalities \cite{pal2009} or ground state energies of many body Hamiltonians  \cite{barthel2012solving}. 
Preliminary experiments not reported in this paper show that relying on CGAL improves some  existing bounds for Bell inequalities, while Mosek runs out memory. 
We intend to improve our software implementation to overcome the accuracy issues arising when using CGAL.
A related investigation track is to design a numerical method for finding the constant trace and the change of basis for noncommutative problems with arbitrary inequality constraints (possibly including nc ball constraint).
Ideally, first order semidefinite solvers  should  have  rich numerical properties when combined with the constant trace and the change of basis matrix obtained in our method.
This will allow us to design, implement and analyze a hybrid numeric-symbolic scheme as in  \cite{pe,multivsos18}, to obtain exact nonnegativity certificates of noncommutative problems. 